\documentclass[11pt]{article}
\usepackage[a4paper,margin=1in]{geometry}
\usepackage[T1]{fontenc}
\usepackage[utf8]{inputenc}
\usepackage{lmodern}
\usepackage{amsmath,amssymb,amsthm,mathtools}
\usepackage{enumitem}
\usepackage{microtype}
\usepackage{authblk}
\usepackage{xcolor}
\usepackage{hyperref}
\hypersetup{colorlinks=true,linkcolor=blue!55!black,citecolor=blue!55!black,urlcolor=blue!55!black,
pdftitle={Lieb--Thirring bounds for Melik--Adamyan canonical Hamiltonians},
pdfsubject={Spectral gap estimates for canonical Hamiltonians},
pdfauthor={Baruch Schneider, Diana Schneiderov\'a, Yifan Zhang},
pdfkeywords={canonical systems, Dirac operators, Lieb--Thirring inequalities, gauge optimization}}
\numberwithin{equation}{section}

\theoremstyle{plain}
\newtheorem{theorem}{Theorem}[section]
\newtheorem{proposition}[theorem]{Proposition}
\newtheorem{lemma}[theorem]{Lemma}
\newtheorem{corollary}[theorem]{Corollary}
\theoremstyle{definition}

\theoremstyle{remark}
\newtheorem{remark}[theorem]{Remark}

\newcommand{\C}{\mathbb C}
\newcommand{\R}{\mathbb R}
\newcommand{\HH}{\mathcal H}
\newcommand{\Tr}{\operatorname{Tr}}
\newcommand{\AC}{\mathrm{AC}}
\newcommand{\loc}{\mathrm{loc}}
\newcommand{\spec}{\sigma}

\newcommand{\Dom}{\operatorname{Dom}}

\newcommand{\Ker}{\operatorname{Ker}}

\title{Lieb--Thirring bounds for Melik--Adamyan canonical Hamiltonians}

\author[1]{Baruch Schneider\footnote{Corresponding author.}}
\author[1]{Diana Schneiderov\'a (Barseghyan)}
\author[1,2,3]{Yifan Zhang}
\affil[1]{Department of Mathematics, University of Ostrava,
702 00 Ostrava, Czechia}
\affil[2]{Department of Algebra, Charles University,
186 75 Prague, Czechia}
\affil[3]{Department of Applied Mathematics,
VSB -- Technical University of Ostrava, 708 00 Ostrava, Czechia}
\date{\small\textit{E-mail addresses:}
\href{mailto:baruch.schneider@osu.cz}{\texttt{baruch.schneider@osu.cz}},
\href{mailto:diana.schneiderova@osu.cz}{\texttt{diana.schneiderova@osu.cz}}, and
\href{mailto:yifan.zhang@osu.cz}{\texttt{yifan.zhang@osu.cz}}}

\begin{document}
\maketitle

\begin{abstract}
We study a class of positive matrix Hamiltonians arising from the canonical differential expressions of Melik--Adamyan and appearing in the appendix of Alpay--Gohberg. Let $J$ and $B$ be self-adjoint involutions on $\mathbb C^{2n}$ satisfying $JB=-BJ$, and let $H>0$ satisfy $HJH=J$. For $m>0$ we consider

$$
\mathcal A_{m,H}=H^{-1}\left(-iJ\frac{d}{dt}+mB\right)
$$

in the weighted space $L^2_H$. A locally absolutely continuous $J$-unitary gauge $\Theta$ representing $H$ reduces this expression to the free massive Dirac operator plus the Hermitian coefficient

$$
P_{m,\Theta}=-i\Theta^*J\Theta'+m(\Theta^*B\Theta-B).
$$

Whenever this coefficient belongs to $L^2$, the corresponding self-adjoint realization, including its operator domain, is independent of the chosen representing gauge. Minimizing $\int\mathrm{Tr}|P_{m,\Theta}|^2$ over the gauge fibre defines an intrinsic energy. A two-sided Birman--Schwinger decoupling, combined with a truncated pseudo-relativistic estimate proved here, gives a $3/2$-moment bound for all eigenvalues in the gap $(-m,m)$ in terms of this energy. The Dirac estimate applies to arbitrary Hermitian matrix coefficients in $L^2$ and requires no sign condition. On the half-line we treat every self-adjoint Lagrangian boundary condition. Two reflection-compatible conditions require no endpoint correction, while an arbitrary condition contributes at most $2nm^{3/2}$. At zero mass, the optimal-gauge energy is computed explicitly in terms of $H^{-1/2}H'H^{-1/2}$. For a scalar hyperbolic-rotation family the massive gauge minimization reduces exactly to a one-dimensional phase functional. We prove existence of a minimizer in the principal phase sector and give an explicit trial phase that strictly and quantitatively improves the positive lift whenever the corresponding first variation is nonzero.
\end{abstract}



\section{Introduction}
Canonical systems provide a common language for one-dimensional spectral
problems.  The differential expressions introduced by Melik--Adamyan in
\cite{Melik1977,Melik1989}, and later used in the appendix of
Alpay--Gohberg \cite{AlpayGohberg1995}, are naturally described by a positive
Hamiltonian
\begin{equation}\label{eq:intro-H}
       \HH=\Theta^{-*}\Theta^{-1},
\end{equation}
where $\Theta$ is a continuous $J$-unitary matrix function.  The historical
formulation only requires continuity.  We work with locally absolutely
continuous Hamiltonians and gauges because the coefficient functional studied
below contains one derivative.

Whenever an admissible massless reduced coefficient belongs to $L^2$, the
corresponding reference operator has essential spectrum equal to the whole
real line and therefore has no finite spectral gap.  We introduce the massive
deformation
\begin{equation}\label{eq:intro-A}
 \mathcal A_{m,\HH}
 =\HH^{-1}\left(-iJ\frac d{dt}+mB\right),\qquad m>0,
\end{equation}
where $B$ anticommutes with $J$.  In every admissible gauge this expression is
unitarily equivalent to a massive Dirac operator with an $L^2$ matrix
coefficient, and its essential spectrum is
$(-\infty,-m]\cup[m,\infty)$.

The main analytic input is a matrix-valued gap estimate.  For a Hermitian
coefficient $Q\in L^2(\R;\C^{2n\times2n})$ we prove
\[
 \sum_{\lambda\in\spec_{\rm d}(D_m+Q)\cap(-m,m)}
       (m-|\lambda|)^{3/2}
 \le C_n m^{1/2}\int_\R\Tr|Q(t)|^2\,dt.
\]
The proof uses the two-sided gap decoupling of Frank--Simon
\cite[Theorem~1.4 and Section~7]{FrankSimon2011}.  Its remaining ingredient is
a truncated bound for
$\sqrt{-d^2/dt^2+m^2}-m-v$.  Unlike the untruncated pseudo-relativistic
estimate used in \cite[Theorem~7.1]{FrankSimon2011}, the truncated moment only
sees eigenvalue depths up to $m$ and is controlled by the single norm
$\|v\|_2$.  This removes the additional $L^{5/2}$ assumption at the exponent
$3/2$.

The structural part of the paper transfers this Dirac estimate to the
Hamiltonian class.  The transfer is not merely formal: we identify the
self-adjoint operator domain, prove independence of the representing gauge,
and retain arbitrary Lagrangian boundary conditions on the half-line.  The
resulting intrinsic coefficient energy is
\begin{equation}\label{eq:intro-energy}
 \mathfrak E_m(\HH)
 =\inf_{\Theta:\,\Theta^{-*}\Theta^{-1}=\HH}
   \int_\R\Tr\left|-i\Theta^*J\Theta'
      +m(\Theta^*B\Theta-B)\right|^2dt,
\end{equation}
where the infimum is taken over gauges for which the displayed coefficient
belongs to $L^2$.  The main full-line and half-line estimates are
Theorems~\ref{thm:main-full} and~\ref{thm:main-half}, respectively.

\subsection*{Relation to previous results}
Three nearby results should be distinguished from the present one.

First, Frank--Simon proved a nearest-gap-edge Lieb--Thirring estimate for a
scalar potential in the one-dimensional $2\times2$ Dirac operator
\cite[Theorem~7.1]{FrankSimon2011}.  At moment exponent $3/2$ their stated
bound involves both $\int |V|^2$ and $\int |V|^{5/2}$.  Our truncated
pseudo-relativistic lemma yields a pure $L^2$ estimate and the argument is
formulated for arbitrary Hermitian $2n\times2n$ coefficients.

Second, Dolbeault, Gontier, Pizzichillo, and Van Den Bosch obtained sharp
Keller estimates and Lieb--Thirring bounds for sign-definite scalar Dirac
potentials; their moment is measured from the upper edge of the gap
\cite[Theorem~1.4]{DolbeaultEtAl2024}.  Our estimate instead controls distance
to the nearer edge, allows both signs and noncommuting matrix coefficients,
and is not an optimizer theorem.

Third, Eckhardt and Kostenko derived sharp $3/2$-moment gap inequalities for
scalar generalized indefinite strings and translated their spectral
characterizations to trace-normalized $2\times2$ canonical systems
\cite[Corollary~10.2, Theorem~14.4, and Appendix~D]{EckhardtKostenko2024}.
Their right-hand side is a weighted scalar coefficient functional obtained
from trace formulas.  The present paper treats the $2n\times2n$
Melik--Adamyan class, uses an unweighted $L^2$ norm of a massive matrix gauge
coefficient, optimizes over the full $J$-unitary gauge fibre, and includes
arbitrary Lagrangian endpoint conditions.  The two inequalities are therefore
complementary rather than competing sharp forms of the same statement.

General variational principles for eigenvalues arising from both ends of a gap
were developed in \cite{DolbeaultEstebanSere2006}.  Recent weak-coupling
asymptotics for one-dimensional massive Dirac operators with Hermitian matrix
potentials were obtained in \cite{AldunateEtAl2026}; those results concern
threshold expansions rather than global eigenvalue moments.
Canonical-system sum rules and distinguished gauges have also been studied in
\cite{BessonovDenisov2023,DamanikEichingerYuditskii2021}.  Our use of gauge
freedom is different: it minimizes a local matrix coefficient energy attached
to a fixed Hamiltonian.

\subsection*{Organization}
Section~2 identifies the gauge fibre and computes the massless energy.
Section~3 constructs the self-adjoint realizations and proves gauge
independence.  Section~4 establishes the full-line matrix Dirac estimate and
then the intrinsic Hamiltonian bound.  Section~5 treats half-line boundary
conditions.  Sections~6 and 7 give the hyperbolic example and the scalar phase
reduction.  The final section records several remaining questions.

\section{Hamiltonians, gauges, and the massless energy}
Throughout the paper, $J,B\in\C^{2n\times2n}$ satisfy
\begin{equation}\label{eq:JB}
 J=J^*,\quad B=B^*,\quad J^2=B^2=I,\quad JB=-BJ.
\end{equation}
The anticommutation implies that the $\pm1$ eigenspaces of $J$ both have
dimension $n$.  After a fixed unitary change of basis one may assume
\begin{equation}\label{eq:standard-JB}
 J=\begin{pmatrix}I_n&0\\0&-I_n\end{pmatrix},\qquad
 B=\begin{pmatrix}0&I_n\\I_n&0\end{pmatrix}.
\end{equation}
We do not otherwise use this normalization except in the boundary discussion
and the example.

A matrix function $\Theta$ is called $J$-unitary if
\begin{equation}\label{eq:Junitary}
       \Theta^*J\Theta=J.
\end{equation}
For such an invertible matrix, set
\begin{equation}\label{eq:HfromTheta}
       \HH=\Theta^{-*}\Theta^{-1}.
\end{equation}
Then
\begin{equation}\label{eq:Hsymplectic}
       \HH>0,\qquad \HH J\HH=J,
       \qquad J\HH J=\HH^{-1}.
\end{equation}
Conversely, every locally absolutely continuous positive Hamiltonian satisfying
\eqref{eq:Hsymplectic} has a locally absolutely continuous $J$-unitary
representative.

\begin{lemma}[Positive lift]\label{lem:positive-lift}
Let $\HH\in\AC_{\loc}(I)$ be positive definite and satisfy
$\HH J\HH=J$.  Then
\begin{equation}\label{eq:Rdef}
       R=\HH^{-1/2}
\end{equation}
is locally absolutely continuous, satisfies $RJR=J$, and represents $\HH$:
$\HH=R^{-*}R^{-1}$.
\end{lemma}

\begin{proof}
The identity $J\HH J=\HH^{-1}$ and functional calculus give
$J\HH^{-1/2}J=\HH^{1/2}$.  Hence
$\HH^{-1/2}J\HH^{-1/2}=J$.  The representation formula is immediate.
Local absolute continuity follows from the smooth matrix functional calculus
on compact intervals on which $\HH$ and $\HH^{-1}$ are bounded.
\end{proof}

For a locally absolutely continuous gauge define
\begin{equation}\label{eq:Qtheta}
       Q_\Theta=-i\Theta^*J\Theta'.
\end{equation}
Differentiating \eqref{eq:Junitary} shows that $Q_\Theta$ is Hermitian, and
\begin{equation}\label{eq:thetaODE}
       \Theta'=\Theta iJQ_\Theta.
\end{equation}
Moreover,
\begin{equation}\label{eq:Hprime}
 \HH'=i\Theta^{-*}(Q_\Theta J-JQ_\Theta)\Theta^{-1}.
\end{equation}
For any matrix $Q$ write
\begin{equation}\label{eq:parallel-perp}
 Q_{\parallel}=\frac12(Q+JQJ),\qquad
 Q_{\perp}=\frac12(Q-JQJ).
\end{equation}
Thus $Q_{\parallel}$ commutes with $J$ and $Q_{\perp}$ anticommutes with $J$.

\begin{proposition}[Horizontal gauge and massless energy]
\label{prop:horizontal}
Let $\HH\in\AC_{\loc}(I)$ satisfy \eqref{eq:Hsymplectic}.  Among all locally
absolutely continuous $J$-unitary gauges representing $\HH$,
\begin{equation}\label{eq:massless-energy}
 \inf_\Theta\int_I\Tr|Q_\Theta|^2dt
 =\frac14\int_I
   \Tr\left[\left(\HH^{-1/2}\HH'\HH^{-1/2}\right)^2\right]dt.
\end{equation}
The infimum is attained by gauges for which
$(Q_\Theta)_{\parallel}=0$.
\end{proposition}

\begin{proof}
If $\widetilde\Theta=\Theta U$, where $U$ is unitary and commutes with $J$,
then $\widetilde\Theta$ represents the same Hamiltonian and
\begin{equation}\label{eq:gauge-law}
 Q_{\widetilde\Theta}=U^*Q_\Theta U-iU^*JU'.
\end{equation}
Conversely, two gauges of the same Hamiltonian differ in this way.  Since $J(Q_\Theta)_{\parallel}$ is Hermitian and commutes with $J$,
the ODE
$U'=-iJ(Q_\Theta)_{\parallel}U$ has a unitary solution commuting with $J$
and removes the parallel component.  Formula
\eqref{eq:Hprime} gives
\[
       \Theta^*\HH'\Theta=-2iJ(Q_\Theta)_{\perp}.
\]
Taking the trace of the square and using $\Theta\Theta^*=\HH^{-1}$ yields
\[
 \Tr|(Q_\Theta)_{\perp}|^2
 =\frac14\Tr\left[\left(
      \HH^{-1/2}\HH'\HH^{-1/2}\right)^2\right].
\]
The two components in \eqref{eq:parallel-perp} are orthogonal in the
Hilbert--Schmidt inner product.  Thus the right-hand side of
\eqref{eq:massless-energy} is a gauge-independent lower bound, and the
horizontal gauge attains it, with equality understood in the extended sense
when either side is infinite.
\end{proof}

\section{Dirac realizations and gauge independence}
Let
\begin{equation}\label{eq:Dm}
       D_m=-iJ\frac d{dt}+mB
\end{equation}
be the free full-line operator on $H^1(\R;\C^{2n})$.  Its spectrum is
$(-\infty,-m]\cup[m,\infty)$.

\begin{proposition}[The $L^2$ matrix Dirac class]\label{prop:L2-Dirac}
If $Q=Q^*\in L^2(\R;\C^{2n\times2n})$, then $D_m+Q$ is self-adjoint on
$H^1(\R;\C^{2n})$.  Moreover,
\[
 \spec_{\rm ess}(D_m+Q)=(-\infty,-m]\cup[m,\infty),
\]
and the spectrum in $(-m,m)$ is discrete.
\end{proposition}

\begin{proof}
For $u\in H^1(\R)$,
\[
 \|Qu\|_2\le \|Q\|_{L^2({\rm op})}\|u\|_\infty
 \le \varepsilon\|u'\|_2+C_\varepsilon
      \|Q\|_{L^2({\rm op})}^2\|u\|_2.
\]
Thus multiplication by $Q$ is infinitesimally $D_m$-bounded, and
Kato--Rellich gives self-adjointness on $H^1$.  For non-real $z$, the Fourier multiplier of $(D_m-z)^{-1}$ is
square-integrable in the momentum variable.  The one-dimensional
Kato--Seiler--Simon estimate therefore makes $Q(D_m-z)^{-1}$
Hilbert--Schmidt.  The perturbation is relatively compact, so Weyl's theorem
gives the essential spectrum statement.
\end{proof}

For an interval $I$, let $\mathcal G(\HH;I)$ be the set of locally absolutely
continuous $J$-unitary gauges satisfying
$\HH=\Theta^{-*}\Theta^{-1}$.  Define
\begin{equation}\label{eq:Pmtheta}
 P_{m,\Theta}=Q_\Theta+m(\Theta^*B\Theta-B).
\end{equation}
This matrix is Hermitian.

\begin{lemma}[Gauge reduction and operator domain]\label{lem:gauge-reduction}
Let $I=\R$ or $I=\R_+$ and let $\Theta\in\mathcal G(\HH;I)$ satisfy
$P_{m,\Theta}\in L^2(I)$.  Multiplication
$U_\Theta f=\Theta f$ is unitary from the unweighted space
$L^2(I)$ onto
\[
 L^2_\HH(I)=\left\{y:\int_I y(t)^*\HH(t)y(t)dt<\infty\right\}.
\]
On the full line, the operator
\begin{equation}\label{eq:canonical-domain}
 \mathcal A_{m,\HH}^{(\Theta)}
 \coloneqq U_\Theta(D_m+P_{m,\Theta})U_\Theta^{-1},\qquad
 \Dom\mathcal A_{m,\HH}^{(\Theta)}=U_\Theta H^1(\R),
\end{equation}
is self-adjoint and acts almost everywhere as
$\HH^{-1}(-iJy'+mBy)$.  On the half-line, a maximal $J$-isotropic boundary
condition $y(0)\in\Lambda$ is transported to
\begin{equation}\label{eq:boundary-transport-general}
 f(0)\in L_\Theta:=\Theta(0)^{-1}\Lambda,
\end{equation}
and the corresponding domain is
$U_\Theta\{f\in H^1(\R_+):f(0)\in L_\Theta\}$.

If $\Theta_0$ and $\Theta_1$ are two gauges with
$P_{m,\Theta_j}\in L^2(I)$, then the transported full-line operators, and the
transported half-line operators with the same canonical boundary $\Lambda$,
coincide.  In particular, both the realization and its domain are gauge
independent.
\end{lemma}

\begin{proof}
The identity $\Theta^*\HH\Theta=I$ proves unitarity.  Since
$\HH^{-1}=\Theta\Theta^*$, direct differentiation on compactly supported
smooth functions gives
\begin{equation}\label{eq:gauge-reduction}
 U_\Theta^{-1}\HH^{-1}
    \left(-iJ\frac d{dt}+mB\right)U_\Theta
 =D_m+P_{m,\Theta}.
\end{equation}
The action formula extends to the transported domain in
\eqref{eq:canonical-domain}.

For gauge independence write $\Theta_1=\Theta_0W$.  Equality of the two
Hamiltonians implies that $W$ is unitary, and $J$-unitarity implies that
$WJ=JW$.  The effective coefficients satisfy
\begin{equation}\label{eq:P-gauge-law}
 P_{m,\Theta_1}=W^*P_{m,\Theta_0}W-iW^*JW'
                  +m(W^*BW-B).
\end{equation}
After solving \eqref{eq:P-gauge-law} for $-iW^*JW'$, the terms involving
$P_{m,\Theta_1}$ and $P_{m,\Theta_0}$ belong to $L^2$, whereas
$W^*BW-B$ is bounded.  Hence $W'=G_2+G_\infty$ with $G_2\in L^2$ and
$G_\infty\in L^\infty$.  If $f\in H^1$, then $G_2f\in L^2$ by
$H^1\hookrightarrow L^\infty$ and $G_\infty f\in L^2$; therefore
multiplication by $W$ maps $H^1$ onto itself.  The same argument applies to
$W^*$.  Equation
\eqref{eq:P-gauge-law} then gives
\[
 D_m+P_{m,\Theta_1}=W^*(D_m+P_{m,\Theta_0})W
 \quad\text{on }H^1,
\]
and $U_{\Theta_1}=U_{\Theta_0}W$, which proves equality of the full-line
realizations.  On the half-line, $W(0)$ maps
$\Theta_1(0)^{-1}\Lambda$ onto $\Theta_0(0)^{-1}\Lambda$, so the same argument
also identifies the boundary domains.
\end{proof}

We henceforth denote the common full-line operator by $\mathcal A_{m,\HH}$.

\section{The full-line gap estimate}
For a self-adjoint operator $A$ with essential gap $(-m,m)$ set
\begin{equation}\label{eq:M32}
 \mathcal M_{3/2}(A)
 =\sum_{\lambda\in\spec_{\rm d}(A)\cap(-m,m)}
       (m-|\lambda|)^{3/2},
\end{equation}
with eigenvalues repeated according to multiplicity.  For a self-adjoint
operator $T$ and a Borel set $I$, write
$N(T\in I)=\Tr\mathbf 1_I(T)$ whenever this number is finite.

Set
\begin{equation}\label{eq:hm-def}
 h_m=\sqrt{-\frac{d^2}{dt^2}+m^2}-m
 \quad\text{in }L^2(\R).
\end{equation}

\begin{lemma}[Truncated pseudo-relativistic estimate]
\label{lem:truncated-pr}
For $v\ge0$ define
\begin{equation}\label{eq:Tmv}
 \mathcal T_m(v)
 =\sum_{\mu\in\spec_{\rm d}(h_m-v)\cap(-\infty,0)}
       \min\{|\mu|,m\}^{3/2}.
\end{equation}
There is a universal constant $C_{\rm pr}$ such that, for every
$v\in L^2(\R)$ with $v\ge0$,
\begin{equation}\label{eq:truncated-pr}
       \mathcal T_m(v)
       \le C_{\rm pr}m^{1/2}\int_\R v(t)^2\,dt.
\end{equation}
One may take
\begin{equation}\label{eq:Cpr}
 C_{\rm pr}=4\left(L_{3/2,1}(\sqrt2-1)^{-1/2}
       +L^{\rm rel}_{1,1}(\sqrt2-1)^{-1}\right),
\end{equation}
where $L_{3/2,1}$ and $L^{\rm rel}_{1,1}$ are any admissible constants in
the one-dimensional Lieb--Thirring inequalities for $p^2$ with moment
$3/2$ and for $|p|$ with moment $1$, respectively.
\end{lemma}

\begin{proof}
We first assume that $v$ is bounded and compactly supported.  Let
$p=-i\,d/dt$, let $P=\mathbf 1_{\{|p|\le m\}}$, and put $c=\sqrt2-1$.
The symbol of $h_m$ obeys
\begin{equation}\label{eq:hm-low-high}
 h_m(p)\ge c\frac{p^2}{m}\quad (|p|\le m),
 \qquad
 h_m(p)\ge c|p|\quad (|p|\ge m).
\end{equation}
Indeed,
$h_m(p)=p^2/(\sqrt{p^2+m^2}+m)$, which gives the first inequality, and
$h_m(p)/|p|$ is increasing for $|p|\ge m$ and equals $c$ at $|p|=m$.

Fix $\tau>0$ and set
\[
 K_\tau=v^{1/2}(h_m+\tau)^{-1}v^{1/2}.
\]
Because $P$ commutes with $h_m$,
\begin{align*}
 K_\tau
 &=v^{1/2}P(h_m+\tau)^{-1}Pv^{1/2}\\
 &\quad+v^{1/2}P^\perp(h_m+\tau)^{-1}P^\perp v^{1/2}.
\end{align*}
The inequalities in \eqref{eq:hm-low-high} imply
\begin{align*}
 P(h_m+\tau)^{-1}P
 &\le P\left(c\frac{p^2}{m}+\tau\right)^{-1}P
 \le\left(c\frac{p^2}{m}+\tau\right)^{-1},\\
 P^\perp(h_m+\tau)^{-1}P^\perp
 &\le P^\perp(c|p|+\tau)^{-1}P^\perp
 \le(c|p|+\tau)^{-1}.
\end{align*}
The Birman--Schwinger principle and the Ky Fan counting inequality
$N(K_1+K_2>1)\le N(K_1>1/2)+N(K_2>1/2)$ therefore give
\begin{equation}\label{eq:pr-count-split}
 N(h_m-v<-\tau)
 \le N\left(c\frac{p^2}{m}-2v<-\tau\right)
      +N(c|p|-2v<-\tau).
\end{equation}

The layer-cake formula now yields
\begin{align*}
 \mathcal T_m(v)
 &=\frac32\int_0^m \tau^{1/2}N(h_m-v<-\tau)\,d\tau\\
 &\le \Tr\left(c\frac{p^2}{m}-2v\right)_{-}^{3/2}
 +\sum_{\substack{\nu<0\\ \nu\in\spec(c|p|-2v)}}
       \min\{|\nu|,m\}^{3/2}.
\end{align*}
The one-dimensional Lieb--Thirring inequality with exponent $3/2$ bounds
the first term.  For the second term we use
$\min\{s,m\}^{3/2}\le m^{1/2}s$ and the moment-one inequality for
$|p|$; see \cite{LiebThirring1976} and \cite[Eq.~(2.14)]{Daubechies1983}.  Scaling gives
\[
 \mathcal T_m(v)
 \le4m^{1/2}\left(L_{3/2,1}c^{-1/2}
             +L^{\rm rel}_{1,1}c^{-1}\right)\int_\R v(t)^2\,dt.
\]

For a general non-negative $v\in L^2(\R)$, choose bounded compactly
supported $v_k\ge0$ with $v_k\to v$ in $L^2$.  All forms
\[
 q_k[u]=\langle u,h_m u\rangle-\int_\R v_k|u|^2,
 \qquad \Dom q_k=H^{1/2}(\R),
\]
are closed and uniformly lower semibounded.  Moreover,
\[
 |q_k[u]-q[u]|
 \le \|v_k-v\|_2\|u\|_4^2
 \le C\|v_k-v\|_2\|u\|_{H^{1/2}}^2.
\]
The form norm of $h_m+1$ is equivalent to the $H^{1/2}$ norm.  Moreover,
multiplication by $v$ is form-compact relative to $h_m$: on every bounded
interval $I$, this follows from the compact embedding
$H^{1/2}(I)\Subset L^4(I)$; see, for example,
\cite[Sections~6 and~7]{DiNezzaPalatucciValdinoci2012}.  The tail is small by
H\"older's inequality and the global continuous embedding
$H^{1/2}(\R)\hookrightarrow L^4(\R)$.  The same applies to every $v_k$.
Thus $q_k\to q$ in form norm and $h_m-v_k\to h_m-v$ in norm resolvent.
For almost every $\tau\in(0,m)$, the point $-\tau$ is not an eigenvalue of
$h_m-v$; at every such $\tau$, the finite-rank spectral projections onto
$(-\infty,-\tau)$ converge in norm and hence in rank.  Applying Fatou's
lemma to the layer-cake formula and using $\|v_k\|_2\to\|v\|_2$ proves
\eqref{eq:truncated-pr}.
\end{proof}

\begin{lemma}[Two-sided reduction to scalar operators]
\label{lem:two-sided-dirac}
Let $Q=Q^*$ be bounded and compactly supported, write
$Q=Q_+-Q_-$, and set $v_\pm(t)=\|Q_\pm(t)\|_{\rm op}$.  Then, for almost
every $E\in(0,m)$ and $\tau=m-E$,
\begin{equation}\label{eq:dirac-to-pr}
 N(D_m+Q\in(-E,E))
 \le2n\left(N(h_m-v_-<-\tau)+N(h_m-v_+<-\tau)\right).
\end{equation}
\end{lemma}

\begin{proof}
The two-sided decoupling formula of Frank--Simon, proved in
\cite[Theorem~1.4]{FrankSimon2011} and applied to the unbounded-below Dirac
operator in the proof of \cite[Theorem~7.2]{FrankSimon2011}, gives
\begin{align}
 N(D_m+Q\in(-E,E))
 &\le N\left(Q_-^{1/2}(D_m-E)^{-1}Q_-^{1/2}>1\right)\notag\\
 &\quad+N\left(Q_+^{1/2}(D_m+E)^{-1}Q_+^{1/2}<-1\right)
 \label{eq:dirac-gap-decoupling}
\end{align}
whenever the endpoints avoid the discrete spectra of the intermediate
operators.  Although the abstract theorem is formulated there for lower-semibounded
operators, the authors apply the resulting decoupling formula directly to the
Dirac operator in their Theorem~7.2.  Their argument also allows noncommuting
$Q_+$ and $Q_-$.  The excluded set of $E$ is discrete and
therefore negligible in the subsequent integration.

Let $r(p)=\sqrt{p^2+m^2}$.  The eigenvalues of the Fourier symbol of $D_m$
are $\pm r(p)$, each with multiplicity $n$.  Consequently,
\begin{equation}\label{eq:free-resolvent-order}
 (D_m-E)^{-1}\le(h_m+\tau)^{-1}I_{2n},
 \qquad
 -(D_m+E)^{-1}\le(h_m+\tau)^{-1}I_{2n}.
\end{equation}
For example, the positive eigenvalue of $(D_m-E)^{-1}$ is
$(r(p)-E)^{-1}=(h_m(p)+\tau)^{-1}$, while its other eigenvalue is negative.

It remains to replace $Q_\pm$ by their operator norms.  Pointwise write
$Q_\pm^{1/2}=C_\pm v_\pm^{1/2}$, where $C_\pm$ are positive multiplication
contractions.  If
$K_\pm=v_\pm^{1/2}(h_m+\tau)^{-1}v_\pm^{1/2}$ on scalar $L^2(\R)$, then the
nonzero eigenvalues of
$C_\pm(K_\pm\otimes I_{2n})C_\pm$ are bounded, with multiplicity, by those
of $K_\pm\otimes I_{2n}$.  Indeed,
$C_\pm(K_\pm\otimes I_{2n})C_\pm
 =((K_\pm^{1/2}\otimes I_{2n})C_\pm)^*
   ((K_\pm^{1/2}\otimes I_{2n})C_\pm)$,
and the assertion follows from the singular-value inequality
$s_j(TC)\le s_j(T)$ for a contraction $C$.  Combining this observation with
\eqref{eq:free-resolvent-order} and the scalar Birman--Schwinger principle
gives
\[
 N\left(Q_\pm^{1/2}(h_m+\tau)^{-1}Q_\pm^{1/2}>1\right)
 \le2nN(h_m-v_\pm<-\tau).
\]
Substitution into \eqref{eq:dirac-gap-decoupling} proves
\eqref{eq:dirac-to-pr}.
\end{proof}

\begin{proposition}[Matrix Dirac estimate]\label{prop:matrix-dirac}
Every Hermitian $Q\in L^2(\R;\C^{2n\times2n})$ satisfies
\begin{equation}\label{eq:matrix-dirac}
 \mathcal M_{3/2}(D_m+Q)
 \le C_n m^{1/2}\int_\R\Tr|Q(t)|^2\,dt,
 \qquad C_n=2nC_{\rm pr}.
\end{equation}
\end{proposition}

\begin{proof}
Assume first that $Q$ is bounded and compactly supported.  The layer-cake
identity
\begin{equation}\label{eq:dirac-layer-cake}
 \mathcal M_{3/2}(D_m+Q)
 =\frac32\int_0^m(m-E)^{1/2}
     N(D_m+Q\in(-E,E))\,dE
\end{equation}
and Lemmas~\ref{lem:two-sided-dirac} and~\ref{lem:truncated-pr} imply
\[
 \mathcal M_{3/2}(D_m+Q)
 \le 2n C_{\rm pr}m^{1/2}\int_\R(v_+^2+v_-^2)\,dt.
\]
Since $v_+^2+v_-^2\le\Tr|Q|^2$ pointwise, this proves
\eqref{eq:matrix-dirac} in the bounded, compactly supported case.

Now let $Q\in L^2$ be arbitrary and choose bounded compactly supported
Hermitian matrices $Q_k\to Q$ in $L^2$.  The infinitesimal relative-bound
estimate in Proposition~\ref{prop:L2-Dirac} gives norm-resolvent convergence
$D_m+Q_k\to D_m+Q$.  For almost every $E\in(0,m)$, neither $E$ nor $-E$ is
an eigenvalue of the limit operator.  At each such $E$, the spectral
projections onto $(-E,E)$ converge in norm and therefore in rank.  Fatou's
lemma in \eqref{eq:dirac-layer-cake}, followed by
$\|Q_k-Q\|_2\to0$, completes the proof.
\end{proof}

Define the optimal-gauge energy
\begin{equation}\label{eq:Em}
 \mathfrak E_m(\HH;I)
 =\inf_{\substack{\Theta\in\mathcal G(\HH;I)\\ P_{m,\Theta}\in L^2(I)}}
       \int_I\Tr|P_{m,\Theta}(t)|^2dt,
\end{equation}
with value $+\infty$ if the admissible set is empty.

\begin{theorem}[Main full-line theorem]\label{thm:main-full}
Let $m>0$ and let $\HH\in\AC_{\loc}(\R)$ be positive definite with
$\HH J\HH=J$.  If $\mathfrak E_m(\HH;\R)<\infty$, then
\begin{equation}\label{eq:ess-full}
 \spec_{\rm ess}(\mathcal A_{m,\HH})
 =(-\infty,-m]\cup[m,\infty)
\end{equation}
and
\begin{equation}\label{eq:main-full}
 \mathcal M_{3/2}(\mathcal A_{m,\HH})
 \le C_n m^{1/2}\mathfrak E_m(\HH;\R).
\end{equation}
\end{theorem}

\begin{proof}
For every admissible gauge, Lemma~\ref{lem:gauge-reduction} and
Proposition~\ref{prop:matrix-dirac} give
\[
 \mathcal M_{3/2}(\mathcal A_{m,\HH})
 \le C_n m^{1/2}\int_\R\Tr|P_{m,\Theta}|^2dt.
\]
The essential spectrum follows from the same unitary reduction and
Proposition~\ref{prop:L2-Dirac}.  Taking the infimum proves
\eqref{eq:main-full}.
\end{proof}

\begin{corollary}[Explicit positive-lift bound]\label{cor:positive-lift}
Let $R=\HH^{-1/2}$ and
\begin{equation}\label{eq:VmH}
 V_{m,\HH}=-iRJR'+m(RBR-B).
\end{equation}
If $V_{m,\HH}\in L^2(\R)$, then
\begin{equation}\label{eq:positive-lift-bound}
 \mathcal M_{3/2}(\mathcal A_{m,\HH})
 \le C_n m^{1/2}\int_\R\Tr|V_{m,\HH}(t)|^2dt.
\end{equation}
If, additionally, $cI\le\HH\le c^{-1}I$ and
$\HH-I,\HH'\in L^2(\R)$, then
\begin{equation}\label{eq:sobolev-bound}
 \mathcal M_{3/2}(\mathcal A_{m,\HH})
 \le C_{n,c}m^{1/2}\int_\R
   \left(\Tr|\HH'|^2+m^2\Tr|\HH-I|^2\right)dt.
\end{equation}
\end{corollary}

\begin{proof}
The positive lift is admissible and its coefficient is
\eqref{eq:VmH}.  Under uniform ellipticity, the integral representation of
the inverse square root and its differentiated form give
\[
 \|R'\|_{\mathfrak S^2}\le C_c\|\HH'\|_{\mathfrak S^2},\qquad
 \|R-I\|_{\mathfrak S^2}\le C_c\|\HH-I\|_{\mathfrak S^2},
\]
where $\|\cdot\|_{\mathfrak S^2}$ denotes the Hilbert--Schmidt norm.
Together with $\|R\|\le c^{-1/2}$, these bounds control the square of
\eqref{eq:VmH} by the integrand in \eqref{eq:sobolev-bound}.
\end{proof}

\begin{corollary}[Zero-mass identity]\label{cor:zero-mass}
For every interval $I$,
\begin{equation}\label{eq:zero-mass}
 \mathfrak E_0(\HH;I)
 =\frac14\int_I
   \Tr\left[\left(\HH^{-1/2}\HH'\HH^{-1/2}\right)^2\right]dt.
\end{equation}
\end{corollary}

\begin{proof}
At $m=0$, $P_{0,\Theta}=Q_\Theta$, so this is
Proposition~\ref{prop:horizontal}.
\end{proof}

\section{Half-line boundary conditions}
Let $I=\R_+$ and let $L\subset\C^{2n}$ be maximal $J$-isotropic:
$u^*Jv=0$ for all $u,v\in L$ and $\dim L=n$.  Such subspaces parametrize the
self-adjoint boundary conditions at the regular endpoint.  Denote by
$D^+_{m,L}+Q$ the operator with domain
\begin{equation}\label{eq:half-domain}
 \Dom(D^+_{m,L}+Q)=\{f\in H^1(\R_+;\C^{2n}):f(0)\in L\}.
\end{equation}

\begin{lemma}[Self-adjoint half-line realizations]\label{lem:halfline-sa}
For every Hermitian $Q\in L^2(\R_+;\C^{2n\times2n})$ and every maximal
$J$-isotropic $L$, the operator $D^+_{m,L}+Q$ is self-adjoint on
\eqref{eq:half-domain}.
\end{lemma}

\begin{proof}
For $Q=0$, integration by parts gives
\[
 \langle D_m^+f,g\rangle-\langle f,D_m^+g\rangle
 =-i f(0)^*Jg(0).
\]
The adjoint boundary space is the $J$-orthogonal complement of $L$, which
equals $L$ by maximal isotropy.  Hence $D^+_{m,L}$ is self-adjoint.  The
half-line Sobolev estimate makes multiplication by $Q$ infinitesimally
$D^+_{m,L}$-bounded, so Kato--Rellich preserves the domain and
self-adjointness.
\end{proof}

Set
\begin{equation}\label{eq:Lsigma}
       L_\sigma=\Ker(B-\sigma I),\qquad \sigma\in\{+1,-1\}.
\end{equation}
These are maximal $J$-isotropic subspaces.

\begin{proposition}[Reflection-compatible boundaries]
\label{prop:reflection}
For Hermitian $Q\in L^2(\R_+;\C^{2n\times2n})$,
\begin{equation}\label{eq:reflection-bound}
 \mathcal M_{3/2}(D^+_{m,L_\sigma}+Q)
 \le 2C_n m^{1/2}\int_0^\infty\Tr|Q(t)|^2dt.
\end{equation}
\end{proposition}

\begin{proof}
Extend $Q$ to $\R$ by
\[
 \widetilde Q(t)=Q(t)\quad(t>0),\qquad
 \widetilde Q(t)=BQ(-t)B\quad(t<0).
\]
The unitary involution $(\mathcal Rf)(t)=Bf(-t)$ commutes with
$D_m+\widetilde Q$.  Its $\sigma$-eigenspace is unitarily equivalent to the
half-line realization with boundary $L_\sigma$.  Hence its gap moment is no
larger than the full-line moment.  Apply Proposition~\ref{prop:matrix-dirac}
and use $\|\widetilde Q\|_2^2=2\|Q\|_2^2$.
\end{proof}

\begin{lemma}[Rank of a boundary resolvent difference]\label{lem:boundary-rank}
Let $Q=Q^*\in L^2(\R_+;\C^{2n\times2n})$ and let $L_0,L_1$ be maximal
$J$-isotropic subspaces.  For non-real $z$,
\begin{equation}\label{eq:boundary-rank}
 \operatorname{rank}\left((D^+_{m,L_1}+Q-z)^{-1}
 -(D^+_{m,L_0}+Q-z)^{-1}\right)\le n.
\end{equation}
\end{lemma}

\begin{proof}
The range of the resolvent difference is contained in the $L^2$ solution
space of
\[
 \left(-iJ\frac d{dt}+mB+Q-z\right)u=0.
\]
The map from this solution space to the initial value $u(0)$ is injective.  If
$\operatorname{Im}z>0$, the Lagrange identity and a sequence $R_k\to\infty$
with $u(R_k)\to0$ give
\[
 u(0)^*Ju(0)=2\operatorname{Im}z\int_0^\infty|u(t)|^2dt>0
 \quad(u\ne0).
\]
Thus the initial values form a $J$-positive subspace of $\C^{2n}$, whose
dimension is at most $n$.  The case $\operatorname{Im}z<0$ is analogous.
\end{proof}

\begin{proposition}[Half-line essential spectrum]
\label{prop:halfline-spectrum}
Let $Q=Q^*\in L^2(\R_+;\C^{2n\times2n})$ and let $L$ be maximal
$J$-isotropic.  Then
\begin{equation}\label{eq:halfline-essential}
 \spec_{\rm ess}(D^+_{m,L}+Q)
 =(-\infty,-m]\cup[m,\infty),
\end{equation}
and the spectrum in $(-m,m)$ is discrete.
\end{proposition}

\begin{proof}
For $L=L_\sigma$, the reflection construction in
Proposition~\ref{prop:reflection} realizes the free half-line operator as a
reducing part of the free full-line Dirac operator.  Compactly supported Weyl
sequences translated to the right show that every point of
$(-\infty,-m]\cup[m,\infty)$ belongs to the essential spectrum of this
reducing part; no point of $(-m,m)$ can belong to it because the direct sum of
the two reflection sectors is the free full-line operator.  Thus
\eqref{eq:halfline-essential} holds for $Q=0$ and $L=L_\sigma$.  It then holds
for every maximal $J$-isotropic $L$ by Lemma~\ref{lem:boundary-rank} and
Weyl's theorem.

Finally, $Q(D^+_{m,L}-z)^{-1}$ is compact for non-real $z$.  Indeed, the
resolvent maps $L^2$ boundedly into $H^1$; on a finite interval, the compact
embedding of $H^1(0,R)$ into $C([0,R])$ makes multiplication by the $L^2$ function $Q$
compact, while the norm of the tail multiplication operator composed with the resolvent is
bounded by $C\|Q\|_{L^2(R,\infty)}$.  Hence $Q$ is relatively compact and
\eqref{eq:halfline-essential} persists.  Discreteness in the open gap follows.
\end{proof}

\begin{lemma}[Finite-rank gap-moment comparison]\label{lem:finite-rank}
Let $A_0,A_1$ be self-adjoint operators with common essential gap $(-m,m)$.
If, for one common resolvent point, their resolvent difference has rank at most
$r$, then, for every $p>0$,
\begin{equation}\label{eq:finite-rank-general}
 \mathcal M_p(A_1)
 \le\mathcal M_p(A_0)+2rm^p,
 \qquad
 \mathcal M_p(A)=\sum_{\lambda\in\spec_{\rm d}(A)\cap(-m,m)}
                   (m-|\lambda|)^p.
\end{equation}
For $p=3/2$ this reads
\begin{equation}\label{eq:finite-rank}
 \mathcal M_{3/2}(A_1)
 \le\mathcal M_{3/2}(A_0)+2rm^{3/2}.
\end{equation}
\end{lemma}

\begin{proof}
Let $z_0$ be a common resolvent point at which the assumed rank bound holds.
The resolvent identity implies, for every other common resolvent point $z$,
\begin{align*}
 &(A_1-z)^{-1}-(A_0-z)^{-1}\\
 &\quad=\bigl[I+(z-z_0)(A_1-z)^{-1}\bigr]
   \bigl[(A_1-z_0)^{-1}-(A_0-z_0)^{-1}\bigr]
   \bigl[I+(z-z_0)(A_0-z)^{-1}\bigr].
\end{align*}
Thus the resolvent difference has rank at most $r$ at every common resolvent
point.

We next prove a counting estimate.  Let $(a,b)$ be compactly contained in the
gap and choose $c\in(a,b)$ outside the discrete spectra of both operators.
Set $R_j=(A_j-c)^{-1}$.  If two bounded self-adjoint operators differ by an
operator of rank at most $r$, then their finite eigenvalue counts above a
threshold lying strictly above the relevant essential spectrum differ by at
most $r$.  To see this,
intersect the relevant spectral subspace with the kernel of the difference and
apply the min--max principle.  The reverse inequality follows after
interchanging the operators.  The analogous assertion below a threshold lying
strictly below the relevant essential spectrum follows by changing signs.

Eigenvalues of $A_j$ in $(c,b)$ correspond under
$\lambda\mapsto(\lambda-c)^{-1}$ to eigenvalues of $R_j$ above
$(b-c)^{-1}$, whereas eigenvalues in $(a,c)$ correspond to eigenvalues of
$R_j$ below $(a-c)^{-1}$.  Both thresholds lie outside the essential spectra
of the resolvents.  Therefore
\begin{equation}\label{eq:count-rank-two}
 \left|N(A_1\in(a,b))-N(A_0\in(a,b))\right|\le2r.
\end{equation}
Possible eigenvalues at the endpoints and at $c$ are handled by monotone
approximation.

Finally, the layer-cake identity
\[
 \mathcal M_p(A)
 =p\int_0^m t^{p-1}N\bigl(A\in(-m+t,m-t)\bigr)dt
\]
and \eqref{eq:count-rank-two} give
\[
 \mathcal M_p(A_1)-\mathcal M_p(A_0)
 \le2rp\int_0^m t^{p-1}dt=2rm^p.
\]
\end{proof}

For a canonical boundary subspace $\Lambda$ in the variable $y(0)$, a gauge
$\Theta$ transports it to
\begin{equation}\label{eq:boundary-transport-short}
       L_\Theta=\Theta(0)^{-1}\Lambda.
\end{equation}
Define the half-line energy by \eqref{eq:Em} with $I=\R_+$ and, for
$\sigma=\pm1$,
\begin{equation}\label{eq:compatible-energy-short}
 \mathfrak E_{m,\sigma}(\HH,\Lambda)
 =\inf_{\substack{\Theta\in\mathcal G(\HH;\R_+)\\
                   \Theta(0)^{-1}\Lambda=L_\sigma\\
                   P_{m,\Theta}\in L^2}}
   \int_0^\infty\Tr|P_{m,\Theta}|^2dt.
\end{equation}

\begin{lemma}[Non-emptiness of the compatible gauge class]
\label{lem:compatible-nonempty}
If $\mathfrak E_m(\HH;\R_+)<\infty$, then, for every maximal
$J$-isotropic $\Lambda$ and either sign $\sigma$,
\[
  \mathfrak E_{m,\sigma}(\HH,\Lambda)<\infty.
\]
\end{lemma}

\begin{proof}
The compact group of unitaries commuting with $J$ acts transitively on the
maximal $J$-isotropic subspaces.  In the standard form
\eqref{eq:standard-JB}, every such subspace is the graph of a unitary
$K:\C^n\to\C^n$, which makes the transitivity explicit.  Start with an
admissible gauge $\Theta_0$.  Choose a unitary $U_0$ commuting with $J$ so that
$U_0^{-1}\Theta_0(0)^{-1}\Lambda=L_\sigma$, and connect $U_0$ smoothly to the
identity inside this compact group on a finite interval.  Taking
$\Theta=\Theta_0U$, the transformation law \eqref{eq:P-gauge-law} shows that
the new terms are bounded and compactly supported.  Hence
$P_{m,\Theta}\in L^2$, and the required boundary condition holds.
\end{proof}

\begin{theorem}[Main half-line theorem]\label{thm:main-half}
Let $m>0$, let $\HH\in\AC_{\loc}(\R_+)$ be positive definite with
$\HH J\HH=J$, and let $\Lambda$ be maximal $J$-isotropic.  Assume that
$\mathfrak E_m(\HH;\R_+)<\infty$.  Then
\begin{equation}\label{eq:canonical-half-essential}
 \spec_{\rm ess}(\mathcal A^+_{m,\HH,\Lambda})
 =(-\infty,-m]\cup[m,\infty),
\end{equation}
and
\begin{equation}\label{eq:main-half}
 \mathcal M_{3/2}(\mathcal A^+_{m,\HH,\Lambda})
 \le 2C_n m^{1/2}\mathfrak E_m(\HH;\R_+)+2n m^{3/2}.
\end{equation}
For either $\sigma=\pm1$,
\begin{equation}\label{eq:main-half-compatible}
 \mathcal M_{3/2}(\mathcal A^+_{m,\HH,\Lambda})
 \le 2C_n m^{1/2}\mathfrak E_{m,\sigma}(\HH,\Lambda).
\end{equation}
\end{theorem}

\begin{proof}
For a fixed admissible gauge, the half-line version of
Lemma~\ref{lem:gauge-reduction} identifies the canonical operator with
$D^+_{m,L_\Theta}+P_{m,\Theta}$.  Proposition~\ref{prop:halfline-spectrum}
gives \eqref{eq:canonical-half-essential}.  By Lemma~\ref{lem:boundary-rank},
changing $L_\Theta$ to $L_+$ gives a resolvent difference of rank at most $n$.
Proposition~\ref{prop:reflection} and Lemma~\ref{lem:finite-rank} therefore
give
\[
 \mathcal M_{3/2}(\mathcal A^+_{m,\HH,\Lambda})
 \le 2C_n m^{1/2}\int_0^\infty\Tr|P_{m,\Theta}|^2dt+2n m^{3/2}.
\]
Taking the infimum proves \eqref{eq:main-half}.  Under the compatibility
constraint $L_\Theta=L_\sigma$, Proposition~\ref{prop:reflection} applies
directly and gives \eqref{eq:main-half-compatible}.
\end{proof}

\begin{remark}[The boundary contribution cannot be removed uniformly]
In the scalar $2\times2$ free case with \eqref{eq:standard-JB}, write a
maximal isotropic boundary line as $u_2(0)=e^{i\theta}u_1(0)$.  A decaying
solution exists precisely when
\[
 e^{i\theta}=\frac{\lambda-i\sqrt{m^2-\lambda^2}}{m}.
\]
For $\theta\in(-\pi,0)$ this gives the gap eigenvalue
$\lambda=m\cos\theta$.  Thus a boundary-uniform absolute moment estimate
cannot have a right side that vanishes with the distributed coefficient.
\end{remark}

\section{A hyperbolic-rotation example}
Take $n=1$ and the matrices in \eqref{eq:standard-JB}.  Let
$a\in H^1(\R)\cap L^\infty(\R)$ be real-valued and set
\begin{equation}\label{eq:example-theta-H}
 \Theta_a(t)=e^{a(t)B},\qquad
 \HH_a(t)=e^{-2a(t)B}.
\end{equation}
Then $\Theta_a$ is positive and $J$-unitary, and $\HH_a J\HH_a=J$.
With
$\sigma_2=\left(\begin{smallmatrix}0&-i\\i&0\end{smallmatrix}\right)$,
\begin{equation}\label{eq:example-P}
 P_{m,\Theta_a}
 =a'\sigma_2+m\left(\sinh(2a)I+(\cosh(2a)-1)B\right).
\end{equation}
The three displayed matrices are Hilbert--Schmidt orthogonal, and therefore
\begin{equation}\label{eq:example-energy-exact}
 \Tr|P_{m,\Theta_a}|^2
 =2|a'|^2+2m^2\left(\sinh^2(2a)+(\cosh(2a)-1)^2\right).
\end{equation}
Consequently,
\begin{equation}\label{eq:example-moment}
 \mathcal M_{3/2}(\mathcal A_{m,\HH_a})
 \le C_1m^{1/2}\int_\R
 \left[2|a'|^2+2m^2\bigl(\sinh^2(2a)+(\cosh(2a)-1)^2\bigr)\right]dt.
\end{equation}

For the concrete profile $a(t)=A\operatorname{sech}(t/L)$, $A,L>0$, use
\[
 \sinh^2(2a)+(\cosh(2a)-1)^2
 =4\cosh(2a)\sinh^2(a)
 \le4\cosh(4A)a^2.
\]
Here $|\sinh a|\le |a|\cosh A$ and
$\cosh(2A)\cosh^2 A\le\cosh(4A)$.
Since
$\int|a'|^2=2A^2/(3L)$ and $\int|a|^2=2A^2L$, we obtain
\begin{equation}\label{eq:example-sech}
 \mathcal M_{3/2}(\mathcal A_{m,\HH_a})
 \le C_1m^{1/2}A^2
 \left(\frac4{3L}+16m^2L\cosh(4A)\right).
\end{equation}
This is a directly evaluable estimate for a non-constant Hamiltonian, without
first recasting the model as a Dirac potential.

\section{Scalar reduction of the massive gauge infimum}
\label{sec:scalar-gauge-reduction}
The preceding example also makes the massive optimization problem explicit.
This is useful because the positive lift need not minimize the energy when
$m>0$.

\begin{proposition}[Exact phase reduction]\label{prop:phase-reduction}
Let $\HH_a=e^{-2aB}$ with real
$a\in H^1(\R)\cap L^\infty(\R)$.  Then
\begin{align}
 \mathfrak E_m(\HH_a;\R)
 =2\inf_{\delta}\int_\R &\Bigl[
   (\delta'+m\sinh(2a))^2 \notag\\
 &+\bigl(m\cosh(2a)\cos(2\delta)-a'\sin(2\delta)-m\bigr)^2 \notag\\
 &+\bigl(m\cosh(2a)\sin(2\delta)+a'\cos(2\delta)\bigr)^2
 \Bigr]dt,\label{eq:phase-functional}
\end{align}
where the infimum is over real-valued phase lifts in $H^1_{\loc}(\R)$ for which
the displayed integral is finite.  For a fixed $\delta$, let
$\mathcal E_a[\delta]$ denote twice the displayed integral.  Every critical phase satisfies, in the weak sense,
\begin{equation}\label{eq:phase-EL}
 \delta''+2ma'\cosh(2a)
 -2m^2\cosh(2a)\sin(2\delta)
 -2ma'\cos(2\delta)=0.
\end{equation}
The positive lift corresponds to $\delta=0$.  It is stationary if and only if
\begin{equation}\label{eq:positive-not-stationary}
       a'(\cosh(2a)-1)=0\quad\text{a.e.}
\end{equation}
If \eqref{eq:positive-not-stationary} fails, a compactly supported phase
perturbation strictly lowers the positive-lift energy.
\end{proposition}

\begin{proof}
Every gauge representing $\HH_a$ is $\Theta_a U$, where $U$ is unitary and
commutes with $J$.  Since the line is simply connected, write
\[
 U=e^{i\alpha}e^{i\delta J}
  =\operatorname{diag}(e^{i(\alpha+\delta)},e^{i(\alpha-\delta)})
\]
with locally $H^1$ real phases.  The Pauli-matrix identities give
\[
 U^*BU=\cos(2\delta)B+\sin(2\delta)\sigma_2,
 \qquad
 U^*\sigma_2U=\cos(2\delta)\sigma_2-\sin(2\delta)B,
\]
and
$-iU^*JU'=\alpha'J+\delta'I$.  Substitution into the gauge law
\eqref{eq:P-gauge-law} and Hilbert--Schmidt orthogonality of
$I,J,B,\sigma_2$ yield
\begin{align*}
 \Tr|P_{m,\Theta_aU}|^2=2\Bigl[&(\alpha')^2
 +(\delta'+m\sinh(2a))^2\\
 &+\bigl(m\cosh(2a)\cos(2\delta)-a'\sin(2\delta)-m\bigr)^2\\
 &+\bigl(m\cosh(2a)\sin(2\delta)+a'\cos(2\delta)\bigr)^2\Bigr].
\end{align*}
The common phase contributes only $2\int(\alpha')^2$ and is minimized by a
constant.  This proves \eqref{eq:phase-functional}.  The first variation with
respect to compactly supported phase perturbations gives
\eqref{eq:phase-EL}.  At $\delta=0$, the first variation in the direction
$\eta$ is
$8m\int a'(1-\cosh(2a))\eta\,dt$.  Hence
\eqref{eq:positive-not-stationary} is the stationarity condition, and otherwise
a perturbation of one sign lowers the energy.
\end{proof}

\begin{theorem}[A principal-sector minimizer]\label{thm:phase-minimizer}
Under the assumptions of Proposition~\ref{prop:phase-reduction}, let
\[
 \mathcal K=\{\delta\in H^1(\R): |\delta(t)|\le\pi/2
                    \text{ for a.e. }t\}.
\]
The phase functional in \eqref{eq:phase-functional} attains its minimum on
$\mathcal K$.  If
\begin{equation}\label{eq:strict-phase-condition}
 a'(\cosh(2a)-1)\not\equiv0,
\end{equation}
then every minimizer is nonzero and its energy is strictly smaller than the
positive-lift energy.  If a minimizer $\delta_*$ satisfies
$\|\delta_*\|_\infty<\pi/2$, then it is a weak solution of
\eqref{eq:phase-EL}.
\end{theorem}

\begin{remark}
Theorem~\ref{thm:phase-minimizer} is deliberately restricted to the
principal $H^1$ phase sector (the zero-winding sector).  It does not assert that the full infimum
in Proposition~\ref{prop:phase-reduction}, which also allows other phase
lifts, is attained.  The explicit phases used below are nevertheless admissible
for the unrestricted infimum and therefore give genuine upper bounds on
$\mathfrak E_m(\HH_a;\R)$.
\end{remark}

\begin{proof}
Write $s=\sinh(2a)$ and $c=\cosh(2a)$.  Subtracting the value at $\delta=0$
from one half of the functional gives the exact identity
\begin{equation}\label{eq:phase-difference}
 \frac12\bigl(\mathcal E_a[\delta]-\mathcal E_a[0]\bigr)
 =\int_\R\left(\delta'^2+2ms\delta'
       +4m^2c\sin^2\delta+2ma'\sin(2\delta)\right)dt.
\end{equation}
Here $s,a'\in L^2$, while $c$ is bounded and $c\ge1$.  On the interval
$[-\pi/2,\pi/2]$ one has
$\sin^2 x\ge4x^2/\pi^2$ and $|\sin(2x)|\le2|x|$.
Young's inequality applied to the two mixed terms in
\eqref{eq:phase-difference} therefore yields constants $c_m>0$ and $C_a<\infty$
such that
\begin{equation}\label{eq:phase-coercive}
 \mathcal E_a[\delta]-\mathcal E_a[0]
 \ge c_m\|\delta\|_{H^1}^2-C_a,
 \qquad \delta\in\mathcal K.
\end{equation}
Thus every minimizing sequence is bounded in $H^1$.

After passing to a subsequence, $\delta_k\rightharpoonup\delta_*$ in $H^1$,
$\delta_k\to\delta_*$ in $L^2_{\rm loc}$, and pointwise a.e.  The set
$\mathcal K$ is weakly closed.  The derivative term in
\eqref{eq:phase-difference} is weakly lower semicontinuous, the term with
$s\delta'$ is weakly continuous, and the non-negative term with
$c\sin^2\delta$ is lower semicontinuous by Fatou's lemma.  Finally,
\[
 \int a'\sin(2\delta_k)\,dt\longrightarrow
 \int a'\sin(2\delta_*)\,dt.
\]
To see the last assertion, use local $L^2$ convergence and then make the tails
small with $a'\in L^2$ and the uniform $H^1$ bound.  Hence the direct method
of the calculus of variations gives a minimizer.

Under \eqref{eq:strict-phase-condition}, the first variation at zero is the
nonzero functional
\[
 8m\int_\R a'(1-c)\eta\,dt,
 \qquad \eta\in C_c^\infty(\R).
\]
A sufficiently small perturbation of the appropriate sign remains in the
interior of $\mathcal K$ and lowers the energy.  Thus zero is not a minimizer
and the minimum is strictly below the positive-lift value.  An interior
minimizer admits arbitrary compactly supported variations, so its first
variation vanishes and gives \eqref{eq:phase-EL}.
\end{proof}

\begin{proposition}[A quantitatively improving phase]
\label{prop:quantitative-phase}
Assume in addition that $a\in H^2(\R)$ and set
\[
 c=\cosh(2a),\qquad \eta=a'(c-1),\qquad
 A_a=\|\eta\|_2^2.
\]
If $A_a>0$, define
\begin{align}
 M_a={}&2\|\eta'\|_2^2
 +8m^2\|c\|_\infty\|\eta\|_2^2
 +8m\|a'\|_2\|\eta\|_4^2,\label{eq:Ma}\\
 \varepsilon_a={}&\min\left\{
      \frac{2mA_a}{M_a},
      \frac{\pi}{4\|\eta\|_\infty}
      \right\}.
 \label{eq:epsa}
\end{align}
Then $\delta_a=\varepsilon_a\eta$ satisfies
$\|\delta_a\|_\infty\le\pi/4$ and
\begin{equation}\label{eq:quantitative-improvement}
 \mathcal E_a[\delta_a]
 \le \mathcal E_a[0]-6mA_a\varepsilon_a.
\end{equation}
In particular, the strict improvement is expressed entirely in terms of the
prescribed Hamiltonian profile $a$.
\end{proposition}

\begin{proof}
The assumption $a\in H^2(\R)$ implies
$\eta\in H^1(\R)\cap L^\infty(\R)$, so all quantities above are finite.  Put
\[
 F(\varepsilon)=\frac12\bigl(
       \mathcal E_a[\varepsilon\eta]-\mathcal E_a[0]\bigr).
\]
Using \eqref{eq:phase-difference} and integrating the term
$2m\int\sinh(2a)\eta'$ by parts gives
\begin{equation}\label{eq:Fprime-zero}
 F'(0)=-4m\int_\R\bigl[a'(c-1)\bigr]^2dt=-4mA_a.
\end{equation}
A second differentiation yields
\[
 F''(\varepsilon)
 =2\|\eta'\|_2^2
  +8m^2\int_\R c\cos(2\varepsilon\eta)\eta^2dt
  -8m\int_\R a'\sin(2\varepsilon\eta)\eta^2dt,
\]
whence $|F''(\varepsilon)|\le M_a$.  Taylor's formula and
$\varepsilon_a\le2mA_a/M_a$ imply
\[
 F(\varepsilon_a)
 \le-4mA_a\varepsilon_a+\frac12M_a\varepsilon_a^2
 \le-3mA_a\varepsilon_a.
\]
The second restriction in \eqref{eq:epsa} gives the asserted sector bound.
Multiplying the last estimate by two proves
\eqref{eq:quantitative-improvement}.
\end{proof}

\begin{remark}
For $a(t)=A\operatorname{sech}(t/L)$, the quantities in
\eqref{eq:Ma}--\eqref{eq:epsa} are finite one-dimensional integrals depending
only on $A$, $L$, and $m$.  Proposition~\ref{prop:quantitative-phase} therefore
turns the qualitative strict improvement into a directly computable bound for
the concrete model in Section~6.
\end{remark}

\begin{remark}
Equation~\eqref{eq:phase-EL} is a forced sine--Gordon-type equation.  The
remaining scalar questions are whether a principal-sector minimizer is always
interior, whether it is unique, and how it decays.  These questions provide a
controlled test problem before attacking the full matrix gauge fibre.
\end{remark}

\section{Concluding remarks and open problems}
The preceding results isolate two mechanisms.  The spectral mechanism is the
pure-$L^2$ matrix Dirac estimate in Proposition~\ref{prop:matrix-dirac}, which
combines two-sided gap decoupling with the truncated pseudo-relativistic bound
of Lemma~\ref{lem:truncated-pr}.  The geometric mechanism is the optimization
of the resulting matrix coefficient over all $J$-unitary representatives of a
fixed Hamiltonian.  The latter is explicit at zero mass and becomes a genuine
variational problem for $m>0$.

Several questions remain open.  First, it would be useful to determine whether
the principal-sector minimizer in Theorem~\ref{thm:phase-minimizer} is always
interior and whether the unrestricted phase problem admits a minimizer.  When
the principal-sector minimizer is interior, it remains to determine whether the
weak solution of \eqref{eq:phase-EL} is unique and has quantitative decay.
Second, the scalar phase equation should be replaced by the matrix Euler--Lagrange system on the
full $U(n)\times U(n)$ gauge fibre.  Third, the universal endpoint term in
\eqref{eq:main-half} may be improvable if the absolute moment is replaced by a
relative moment with respect to the free realization having the same boundary
condition.  Finally, sharp constants and possible trace-formula refinements of
\eqref{eq:main-full} remain to be investigated.

\section*{Conflict of interest}
The authors declare that they have no competing interests related to the publication of this paper.

\section*{Authors’ contributions}
All authors contributed equally to the manuscript and approved the final version.

\medskip
\noindent\textbf{Funding.}
This work was co-funded by the University of Ostrava, Grant No.~SGS05/P\v{R}F/2026. Yifan Zhang is also co-funded by the European Union under the REFRESH -- Research Excellence For REgion Sustainability and High-tech Industries project, No.~CZ.10.03.01/00/22003/0000048, via the Operational Programme Just Transition.

\end{document}